\documentclass[11pt, reqno]{amsart}
\title{Optimal matchings of randomly perturbed lattices}
\date{June, 2025}

\author[Dor Elboim]{Dor Elboim}
\author[Yinon Spinka]{Yinon Spinka}
\author[Oren Yakir]{Oren Yakir}
\address{Dor Elboim \newline Department of Mathematics, Stanford University, USA.}
\email{dorelboim@gmail.com}
\address{Yinon Spinka \newline School of Mathematical Sciences, Tel Aviv University, Israel. }
\email{yinonspi@tauex.tau.ac.il}
\address{Oren Yakir \newline Department of Mathematics, Stanford University, USA.}
\email{oren.yakir@gmail.com}
    
\usepackage{amsmath,amsfonts,amssymb,amsthm,graphics,enumerate,mathtools}
\usepackage{hyperref,dsfont}
\hypersetup{
    colorlinks=true,
    linkcolor=blue,
    citecolor=red,
    urlcolor=blue,
    pdfborder={0 0 0}
}

\usepackage{graphicx,xcolor,bbm,fullpage}
\usepackage{geometry}
\newgeometry{vmargin={31 mm}, hmargin={26mm,26mm}}
\usepackage{cleveref}
  \crefname{theorem}{Theorem}{Theorems}
  \crefname{thm}{Theorem}{Theorems}
  \crefname{lemma}{Lemma}{Lemmas}
  \crefname{lem}{Lemma}{Lemmas}
  \crefname{remark}{Remark}{Remarks}
  \crefname{prop}{Proposition}{Propositions}
  \crefname{proposition}{Proposition}{Propositions}
  \crefname{notation}{Notation}{Notations}
  \crefname{claim}{Claim}{Claims}
  \crefname{observation}{Observation}{Observations}
  \crefname{defn}{Definition}{Definitions}
  \crefname{corollary}{Corollary}{Corollaries}
  \crefname{section}{Section}{Sections}
  \crefname{figure}{Figure}{Figures}
  \crefname{exercise}{Exercise}{Exercises}
    \crefname{assumption}{Assumption}{Assumptions}

\newtheorem{thm}{Theorem}[section]

\newtheorem{claim}[thm]{Claim}
\newtheorem{lemma}[thm]{Lemma}
\newtheorem{corollary}[thm]{Corollary}
\newtheorem{prop}[thm]{Proposition}

\numberwithin{equation}{section}

\theoremstyle{definition}
\newtheorem{remark}[thm]{Remark}

\def\cQ{\mathcal{Q}}

\def\cF{\mathcal{F}}
\def\cE{\mathcal{E}}
\def\cD{\mathcal{D}}
\def\cC{\mathcal{C}}

\def\P{\mathbb{P}}
\def\E{\mathbb{E}}

\def\R{\mathbb{R}}
\def\Z{\mathbb{Z}}
\def\N{\mathbb{N}}

\def\G{\mathbb{G}}

\def\R{\mathbb{R}}

\newcommand{\1}{\mathds{1}}
\def\eps{\varepsilon}

\DeclareMathOperator{\var}{Var}

\DeclareMathOperator{\diam}{diam}

\begin{document}
\maketitle

\begin{abstract}
Consider a point process in Euclidean space obtained by perturbing the integer lattice with independent and identically distributed random vectors. Under mild assumptions on the law of the perturbations, we construct a translation-invariant perfect matching between this point process and the lattice, such that the matching distance has the same tail behavior as the hole probability of the point process, which is a natural lower bound.

\end{abstract}

\section{Introduction}
Consider the Gaussian perturbations of the lattice $\Z^2$, that is
\begin{equation}
    \label{eq:intro_2D_gaussian_perturbations}
    \Pi = \{ v + \xi_v \, : \, v\in \Z^2 \} \, ,
\end{equation}
where $\{\xi_v\}_{v\in \Z^2 }$ are i.i.d.\ Gaussian random vectors in $\R^2$. Suppose we want to ``assign back" the random set $\Pi$ to the lattice $\Z^2$ in a translation invariant way. The first scheme that comes to mind is $T_0:\Pi \to \Z^2$, where
\[
T_0(v + \xi_v) = v \, . 
\]
That is, each perturbed point is sent back to the lattice point it originated from. Clearly,
$$\log \P \big( \|T_0^{-1}(0) \| \ge r\big) = \log \P \big( \|\xi_0 \| \ge r\big) \simeq - r^2 \, .$$ 
Here and throughout, $A\simeq B$ means that both $A= O(B)$ and $B= O(A)$ as $r\to\infty$. 
However, $T_0$ does not necessarily minimize the tail behavior of $ |T^{-1}(0) |$ among all possible translation invariant bijections $T:\Pi \to \Z^2$, and a different scheme could lead to faster tail decay. Denote by $B_r$ the ball of radius $r$ centered at the origin. A simple computation (see Lemma~\ref{lemma:hole_probability_Z_d} below) shows that
\begin{equation}
    \label{eq:intro_hole_probability_to_2D_gaussian}
    \log \P\big(\Pi \cap B_r = \emptyset \big) \simeq -r^4 \, .
\end{equation}
On the event $\{\Pi \cap B_r = \emptyset\}$, any bijection $T$ must have $ \|T^{-1}(0) \| \ge r$, so~\eqref{eq:intro_hole_probability_to_2D_gaussian} gives a lower bound of $\exp (-Cr^4)$ on the optimal tail, for some $C>0$. As a special case of the main result of this paper (Theorem~\ref{thm:invariant_matching_Z_d}), we show that this lower bound is tight up to the constant in the exponent:

\vspace{2mm}

\begin{center}
    For $\Pi$ given by~\eqref{eq:intro_2D_gaussian_perturbations}, there exists a translation invariant bijection ${\sf T}:\Pi\to \Z^2$ so that $\P\big(  \|{\sf T}^{-1}(0) \| \ge r \big) \le \exp\big(-c r^4\big)$, for some constant $c>0$ and all sufficiently large $r$. 
\end{center}

\subsection{Random perturbations of the lattice}

In this paper we study the problem of finding an invariant bijection with optimal tails between the lattice and its random perturbations. For $d\ge 1$ we consider the \textbf{random perturbations of the lattice}, which is the point process\footnote{When the distribution of the perturbation $\xi$ is continuous, we get a simple point process, but in general, this point process may have multiplicities.} 
\begin{equation}
    \label{eq:def_of_random_perturbations_of_the_lattice}
    \Pi = \{ v + \xi_v \, : \, v\in \Z^d \},
\end{equation}
where $\{\xi_v\}$ is a sequence of i.i.d.\ random vectors in $\R^d$. Note that~\eqref{eq:intro_2D_gaussian_perturbations} is a special case of~\eqref{eq:def_of_random_perturbations_of_the_lattice}, with $d=2$ and Gaussian perturbations. Beyond Gaussian perturbations, another class of interest are polynomial perturbations, i.e.\ where $\P\big( \|\xi \| \ge r\big) \simeq r^{-\alpha}$ for some $\alpha>0$ (we write $\xi$ for a random vector with the same distribution as the $\xi_v$). In general, we do not assume that $\xi$ is spherically symmetric, though this is a natural situation to have in mind. 

A \textbf{matching} between two countable sets $\Lambda$ and $\Xi$ is an injective mapping $M \colon \Lambda \to \Xi$. We call the matching \textbf{perfect} if $M$ is onto $\Xi$, in which case, the inverse mapping $M^{-1} \colon \Xi \to \Lambda$ is matching between $\Xi$ and $\Lambda$. An \textbf{invariant matching} between $\Z^d$ and $\Pi$ is a random matching $M : \Z^d \to \Pi$, such that the joint law of $(\Pi,M)$ is invariant under translations by elements of $\Z^d$. Every perturbed lattice comes with a canonical perfect invariant matching, namely $M_0(v) = v+\xi_v$. Given any invariant matching $M$, we consider the matching distance $\|M(v)-v\|$ of a point $v \in \Z^d$, the distribution of which is the same for all $v$. We are particularly interested in its tail behavior, namely, the rate of decay of $\P (\|M(0)\| \ge r)$ as $r \to \infty$. 

Define
\[ p(r) = \P \big(\|\xi\| \ge r\big) \, , \quad \text{and} \quad h(r) = \P\big(\Pi\cap B_r = \emptyset\big) \, .\]
The optimal tail behavior for an invariant matching lies between $h(r)$ and $p(r)$. Indeed, the canonical matching $M_0$ has matching distance tail $p(r)$, while on the ``hole" event $\{\Pi\cap B_r = \emptyset\}$ we must have $\|M(0)\|\ge r$ for any matching $M:\Z^d\to \Pi$. To better understand this gap, it is convenient to try to express (bound above and below) $h(r)$ in terms of $p(r)$. We do this later, but for now only mention that as long as the perturbation are ``nice", we have
\[ r^d \log p(r) \simeq \log h(r) . \]
In particular,
\begin{align*}
 \log h(r) &\simeq -r^{d+2} &&\text{for any Gaussian perturbation},\\
 \log h(r) &\simeq -r^d \log r &&\text{for any polynomial perturbation}.
\end{align*}

\subsection{Main result}

Throughout, we will assume that the perturbation distribution function $p(r)$ is sufficiently regular. Specifically, we will assume there exists $C>0$ so that
\begin{equation}\label{eq:assumption i}
   \int_{r}^{\infty} p(t) \, {\rm d} t \le C  r  p(r)   \tag{Int} 
\end{equation}
for all $r\ge 1$. It is not hard to check that~\eqref{eq:assumption i} holds for Gaussian perturbations, as well as for polynomial perturbations with $\alpha>1$. For polynomial perturbations with $\alpha\le 1$, the integral on the left-hand side of~\eqref{eq:assumption i} diverges, and hence~\eqref{eq:assumption i} does not hold in this case. We will further impose on the perturbation distribution that
\begin{equation}\label{eq:assumption ii}
  \sup_{r\ge 1} \frac{\log p(kr)}{\log p(r)}<\infty \quad\text{for all }k\ge1. \tag{Reg}
\end{equation}
The assumption~\eqref{eq:assumption ii} is not essential for our result to hold, but it simplifies the presentation. For the most part, it allows a cleaner description of the hole probability $h(r)$ in terms of $p(r)$; see \cref{lemma:hole_probability_Z_d} and \cref{rem:assumptions} below. Again, it is not hard to verify that~\eqref{eq:assumption ii} holds for Gaussian and polynomial perturbations. Our main result is that the hole probability is attainable by a perfect invariant matching for a large class of random perturbations. 

\begin{thm}
\label{thm:invariant_matching_Z_d}
 Let $\Pi$ be a perturbed lattice~\eqref{eq:def_of_random_perturbations_of_the_lattice}, whose perturbation $\xi$ satisfies \eqref{eq:assumption i} and \eqref{eq:assumption ii}. There exists $c>0$ and a perfect invariant matching $M$ between $\Z^d$ and $\Pi$ whose matching distance satisfies, for all sufficiently large $r$,
 \[\P (\|M(0)\| \ge r) \le  h(r)^c \, .\]
\end{thm}
Theorem~\ref{thm:invariant_matching_Z_d} implies, in particular, that there exists an invariant perfect matching whose matching distance tail attains the hole probability for any Gaussian perturbation, and for any polynomial perturbation with $\alpha>1$. We give the proof of Theorem~\ref{thm:invariant_matching_Z_d} in Section~\ref{sec:invariant_matching_perturbed_lattice} below. 

Recall that a matching $M:\Z^d \to \Pi$ is called a \textbf{factor} (of $\Pi)$ if there exists a deterministic $\Z^d$-equivariant function $f$ such that $M=f(\Pi)$. We remark that the perfect matching prescribed in Theorem~\ref{thm:invariant_matching_Z_d} is \emph{not} a factor, as the proof relies at several points on compactness arguments, which in principle introduce additional randomness beyond that of $\Pi$ itself. We do not know whether it is possible to construct a perfect factor matching between $\mathbb{Z}^d$ and $\Pi$ which attains the hole probability; see Section~\ref{sec:open_problems} for a discussion on this problem.

It is also interesting to study polynomial perturbations with $\alpha\le 1$. For $d=1$, we prove that for $\alpha\in(0,1)$ there is a translation invariant matching with tail  $\mathbb P (|M(0)|\ge r) \le C r^{-\frac{1+\alpha }{2}}$, and that this cannot be improved by much.

\begin{thm}\label{thm:d=1}
    Fix $d=1$ and $\alpha\in(0,1)$. Suppose that the distribution of the perturbation $\xi$ is symmetric around the origin and satisfies 
    \begin{equation}
       cr^{-\alpha } \le  \mathbb P (|\xi | \ge r) \le Cr^{-\alpha }
    \end{equation}
    for two constants $C,c>0$ and all $r \ge 1$. Then:
    \begin{enumerate}
        \item 
        There is a perfect matching $M:\mathbb Z \to \Pi$ such that $$\mathbb P (|M(0)|>r)\le Cr^{-\frac{1+\alpha }{2}} \, .$$
        \item 
        Any invariant perfect matching $M:\mathbb Z \to \Pi $ has $$\E \big[ |M(0)|^{\frac{1+\alpha}{2}} \big] = \infty \, .$$
    \end{enumerate}
\end{thm}
Note that the hole probability in this case is $\exp(-\Theta(r\log r))$, which is much smaller than $r^{-\frac{1+\alpha }{2}}$. The proof of Theorem~\ref{thm:d=1} is provided in Section~\ref{sec:polynomial_in_d_1}. 
We do not expect that this transition in $\alpha$ persists for polynomial perturbations in $d\ge 2$, and instead we believe that the hole probability is attainable in $d=2$ for \emph{any} polynomial perturbations. We also mention that the matching constructed in the first part of the theorem is a factor matching scheme. For further discussion on these points and other open problems, see Section~\ref{sec:open_problems}.

\subsection{Related works}
\label{sec:related_works}

A closely related line of work concerns the problem of matching a unit intensity Poisson process to the lattice $\Z^d$. In this case, the tail behavior for the matching distance is strongly dimension-dependent. It is shown in~\cite{Holroyd-Pemantle-Peres-Schramm} that for any translation invariant matching from the Poisson process to the lattice we have 
\[
\E[ \|M(0)\|^{d/2}] = \infty \, , \qquad \text{for} \ d\in\{1,2\}\, , 
\]
while for $d\ge 3$ it is possible to construct a matching with $\|M(0)\|$ having finite exponential moment. See also~\cite{Angel-Ray-Spinka, Chatterjee-Peled-Peres-Romik, Hoffman-Holroyd-Peres, Talagrand, timar2023factor} for related background and results. 

A point process $\Pi\subset \R^d$ is called \textbf{hyperuniform} if 
\[
\lim_{r\to \infty} \frac{\var \big(|\Pi\cap B_r|\big)}{r^d} = 0 \, .
\]
In particular, the Poisson process is \emph{not} hyperuniform. The concept of hyperuniformity, introduced in~\cite{Torquato, Torquato-Stillinger}, captures a balance between local randomness and large-scale order, reminiscent of a lattice. Random perturbations of the lattice are perhaps the simplest example of a hyperuniform point process~\cite{Gacs-Szasz,Yakir-fluctuations}. Recently, it was shown~\cite{LR-Yogeshwaran} that in $d=2$, any translation-invariant hyperuniform point process with a mild integrability condition can be matched to the lattice $\Z^2$ with matching distance having finite second moment. Moreover, this condition is almost sharp in $d=2$, and fails completely for $d\ge3$, as shown in the recent work~\cite{Dereudre-Flimmel-Huesmann-Leble}. This raises a natural question: for particularly well-behaved hyperuniform point processes, such as the two-dimensional Gaussian perturbations of the lattice~\eqref{eq:intro_2D_gaussian_perturbations} -- can stronger tail decay bounds be obtained on the matching distance (e.g., a finite exponential moment)? 
Theorem~\ref{thm:invariant_matching_Z_d} provides a positive answer to this question for a broad class of random perturbations of the lattice, including the Gaussian case.  

Besides being a canonical example of hyperuniform point processes, random perturbations of the lattice are a well-studied model, arising in several other physical context. We do not aim to survey the full scope of the physics literature, and just note that perturbed lattices appear as models of thermal fluctuations (see, e.g., \cite[Section~5.3]{Baake-Grimm}) and have been proposed as a framework for modeling cosmic structure~\cite{Gabrielli-Joyce-Labini}. In mathematics, there has also been a recent resurgence in studying different aspects of perturbed lattices, see for instance~\cite{Flimmel,Mastrilli, Peres-Sly, Yakir-recovering}.

\subsubsection*{Acknowledgments}
We thank Mikhail Sodin for helpful discussions. Y.S.\ is supported in part by ISF grant 1361/22.
O.Y.\ is supported in part by NSF postdoctoral fellowship DMS-2401136.

\section{Invariant matching for the perturbed lattice}
\label{sec:invariant_matching_perturbed_lattice}

In this section we prove \cref{thm:invariant_matching_Z_d}.
We first provide in \cref{sec:deterministic_matching_condition} a deterministic matching condition based on notions of ``cover'' and ``crossing''.
We next construct in \cref{sec:regular-cover} a random cover suited to our particular point process.
We then provide in \cref{sec:cover-prob-estimates} probabilistic estimates on this random cover. 
Finally, we put all of these steps together in \cref{sec:thm-proof} to prove \cref{thm:invariant_matching_Z_d}.

When proving Theorem~\ref{thm:invariant_matching_Z_d} we may work with the $\ell^\infty$ distance on $\R^d$. Indeed, this can be seen using Lemma~\ref{lemma:hole_probability_Z_d} (which is stated for $\ell ^{\infty}$) and assumption~\eqref{eq:assumption ii}. Hence throughout this section we use the $\ell^\infty$ distance, so that balls are Euclidean boxes.

\subsection{A deterministic matching condition}
\label{sec:deterministic_matching_condition}

In the following proposition we construct a matching between $\mathbb Z ^d$ and a deterministic set of points in $\mathbb R^d$ that satisfy some requirements. We start with some definitions.

A \textbf{regular cover} of $\mathbb R^d$ is a countable collection $\cD$ of compact subsets of $\mathbb R ^d$ whose union is all of $\mathbb R ^d$, and such that each $v \in \mathbb Z ^d$ is contained in a unique set $D \in \cD$.  We write $D_v$ for this set, and we denote by $N_v$ the union of all $D \in \cD$ that intersect $D_v$. Let us note that in our application of the proposition below, the cover $\mathcal D$ will simply contain blocks in $\mathbb R^d$.

Let $(\Pi_v)_{v \in \mathbb Z ^d}$ be a locally finite collection of points in $\mathbb R^d$ indexed by $\mathbb Z ^d$. Let $\Pi$ be the set of these points (with multiplicities).
Given a subset $D \subset \R^d$, we say that $u \in \mathbb Z ^d$ \textbf{crosses} $D$ if the line segment $[u,\Pi_u]$ intersects $D$, but $\Pi_u \notin D$. Let $\cC(D)$ be the set of $u \in \Z^d$ that cross $D$.

\begin{prop}\label{prop:deterministic}
Let $\mathcal D$ be a regular cover and let $(\Pi_v)_{v \in \mathbb Z ^d}$ be locally finite.
Suppose that
\[ |\cC(D)| \le |D \cap \Z ^d| \qquad\text{for all }D \in \cD .\]
Then there exists a matching $M \colon \mathbb Z ^d \to \Pi$ such that $M(v) \in N_v$ for all $v \in \Z^d$.
\end{prop}

\begin{proof}
By compactness, it is enough to prove that for any finite set $L \subseteq \Z ^d$ there is a matching between $L$ and $\Pi$ such that each $v \in L$ is matched with a point in $N_v$.

Let $\G$ be the bipartite graph on $L \cup \Pi$ in which $v \in L$ and $x \in \Pi$ are adjacent if $x \in N_v$.
By Hall's theorem, the desired matching exists as long as $|N(A)| \ge |A|$ for all $A \subseteq L$.
Here, $N(A)$ is to be understood as a multiset, and $|N(A)|$ counts points with their multiplicities. 

For a fixed $A \subset L$, we define
\[ \cE := \{ D \in \cD : D\cap A \neq \emptyset \}, \quad E := \bigcup_{D\in \cE} D \, , \quad \cF := \{ D \in \cD \setminus \cE : D \cap E \neq \emptyset\}, \quad F := \bigcup_{D\in \cF} D \,  .\]
Note that
\[ N(A) = \Pi \cap (E \cup F) .\]
Next, define the ``escaping'' points via
\[ B := \{ v \in (E \cup F) \cap \Z ^d : \Pi_v \notin N(A) \} = \{ v \in (E \cup F) \cap \Z ^d : \Pi_v \notin E \cup F \} \, ,\]
and observe that
\[ B \subset \bigcup_{D \in \cF} \cC(D) .\]
Indeed, if $v\in F \cap \Z ^d$ and $\Pi _v\notin E\cup F$ then $v$ crosses $D_v\in \mathcal F$. Moreover, if $v\in E \cap \Z ^d$ and $\Pi _v\notin E\cup F$, then the line segment $[v,\Pi_v]$ intersects some $D \in \cF$ (since $E \cup (\R ^d \setminus F)$ is not connected) and $\Pi _v\notin D$, so that $v$ crosses $D$.
Therefore,
\[ |B| \le \sum_{D \in \cF} |\cC(D)| \le \sum_{D \in \cF} |D \cap \Z ^d| = |F \cap \Z ^d| ,\]
where the last equality uses that each $v \in \Z ^d$ is contained in a unique $D \in \cD$.
Thus,
\begin{align*}
|N(A)| = |\Pi \cap (E \cup F)| &\ge |\{v \in (E \cup F) \cap \Z ^d : \Pi_v \in E \cup F\}| \\&= |(E \cup F) \cap \Z ^d| - |B| \ge |E \cap \Z ^d| \ge |A| . \qedhere
\end{align*}
\end{proof}

\begin{remark}
    In the last proposition, the lattice $\Z^d$ can be replaced by any locally finite subset of $\R^d$.
    In fact, it is clear from the proof that only minor topological properties of the underlying space are needed, and that analogous results hold on spaces other than $\R^d$.
\end{remark}

\subsection{Constructing the regular cover}
\label{sec:regular-cover}

Recall the definitions from \cref{sec:deterministic_matching_condition}.
Let $(\Pi_v)_{v \in \Z^d}$ be a point set in $\R^d$.
Our goal now is to construct a regular cover $\cD$ of $\R^d$ such that
\begin{equation}\label{eq:D-prop1}
    |\cC(D)| \le |D \cap \Z^d| \qquad\text{for all }D \in \cD
\end{equation}
and
\begin{equation}\label{eq:D-prop2}
    \diam D \le 2 \diam D' \qquad\text{for all $D,D' \in \cD$ that intersect.}
\end{equation}

Our regular cover $\cD$ will consist of closed boxes whose interiors are pairwise disjoint. Moreover, all the boxes will be diadic (shifted by $\frac12$), i.e., of the form $k+[-\frac12,2^i-\frac12]^d$ for some $i \in \N$ and $k\in 2^i \Z^d$. Let $\cQ_i$ be the set of boxes of the form $k+[-\frac12,2^i-\frac12]^d$, where $k\in 2^i \Z^d$. For $v\in \Z^d$, let $Q_i(v)$ be the unique box $Q\in \cQ_i$ containing $v$. Denote $q_i := |Q_i(v) \cap \Z^d| = 2^{id}$.

We construct $\cD$ in several steps. First, define
\begin{equation}\label{eq:I_0}
    I^0_v := \min \big\{ i : |\cC(Q_i(v))| \le q_i \big\} \quad \text{and} \quad R^0_v := 2^{I^0_v} .
\end{equation}
Next, we smoothen out the field $(I^0_v)_{v \in \Z^d}$ by defining
\begin{equation}\label{eq:I_1}
    R^1_v := \max \big\{ R^0_u - \tfrac14\|u-v\| : u \in \Z^d \big\} \quad \text{and} \quad
    I^1_v := \lceil \log_2(R^1_v) \rceil .
\end{equation}
Next, we define
\begin{equation}\label{eq:I}
    I_v := \max \big\{ I^1_u : u \in \Z^d,~ v \in Q_{I^1_u}(u) \big\} \quad \text{and} \quad
    R_v := 2^{I_v} .
\end{equation}
Finally, we let $\cD$ consist of all $Q_{I_v}(v)$, $v \in \Z^d$ (without repetition).

\begin{lemma}\label{lem:construction of D}
    If $R_v$ is finite for all $v$, then $\cD$ is a regular cover satisfying \eqref{eq:D-prop1} and \eqref{eq:D-prop2}.
\end{lemma}
\begin{proof}
Let us first check that $\cD$ is a regular cover. Clearly, the union of the sets in $\cD$ is all of $\R^d$, since for $v\in \mathbb Z ^d$ we have $v+[-\frac12,\frac12]^d=Q_0(v) \subseteq Q_{I_v}(v)$. Let us check that each element of $\Z^d$ is contained in a unique $D \in \cD$ (since all our boxes are unions of sets of the form $v+[-\frac12,\frac12]^d$ with $v \in \Z^d$, this is actually equivalent to the interiors of the $D \in \cD$ being pairwise disjoint). Indeed, if the interiors of $Q_{I_v}(v)$ and $Q_{I_w}(w)$ intersect, then since all our boxes are diadic, one must be contained in the other, say, $Q_{I_v}(v) \subset Q_{I_w}(w)$. However, using \eqref{eq:I} we have that $I_w=I^1_u$ for some $u \in \Z^d$ with $w \in Q_{I^1_u}(u)$, then also $v \in Q_{I^1_u}(u)$ so that $I_v \ge I_w$, and hence $I_v=I_w$ and $Q_{I_v}(v)=Q_{I_w}(w)$.

Let us now show \eqref{eq:D-prop2} holds, namely, that $\diam D \le 2 \diam D'$ whenever $D,D' \in \cD$ intersect. Since $D=Q_{I_v}(v)$ for all $v \in D \cap \Z^d$, and similarly for $D'$, we may choose $v,w \in \Z^d$ such that $D=Q_{I_v}(v)$, $D'=Q_{I_w}(w)$ and $\|v-w\|=1$. We need to show that $I_v \le I_w+1$. By \eqref{eq:I}, there is $u \in \Z^d$ such that $v \in Q_{I^1_u}(u)$ and $I_v=I^1_u$. We obtain $\|u-w\| \le \|u-v\|+1 \le 2^{I^1_u}\le 2R_u^1$, where the last inequality is by \eqref{eq:I_1}. We need to show that $I_v=I^1_u \le I_w+1$. Since $I_w \ge I^1_w$, it suffices to show that $I^1_u \le I^1_w+1$. To this end, it suffices to show that $\log_2(R^1_u) \le \log_2(R^1_w)+1$, or equivalently, $R^1_u \le 2R^1_w$.
Since $\|u-w\| \le 2R^1_u$ (as shown above), this will follow if we prove that $R^1_u - R^1_w \le \frac14\|u-w\|$. To this end, recall that by \eqref{eq:I_1} there exists $z \in \Z^d$ such that $R^1_u = R^0_z-\frac14\|z-u\|$, and also $R^1_w \ge R^0_z-\frac14\|z-w\|$. Thus, by the triangle inequality
\[ R^1_u - R^1_w \le \tfrac14\|z-w\| - \tfrac14\|z-u\| \le \tfrac14 \|u-w\| .\]
Finally, we show the \eqref{eq:D-prop1} holds, namely, that $|\cC(D) \le |D \cap \Z^d|$ for all $D \in \cD$.
We prove something stronger, namely, that $|\cC(Q_i(v))| \le q_i$ for all $v \in \Z^d$ and $i \ge 0$ for which $I^0_u \le i$ for all $u \in Q_i(v)$. This is indeed stronger since $I^0_u \le I^1_u \le I_v$ for all $u \in Q_{I_v}(v)$.
We prove the claim by induction on $i$. The base case $i=0$ is trivial from the definition of $I^0_v$ in \eqref{eq:I_0}.
Suppose that $i \ge 1$ and $v \in \Z^d$ are such that $I^0_u \le i$ for all $u \in Q_i(v)$.
If $I^0_u = i$ for some $u \in Q_i(v)$, then $|\cC(Q_i(v))| = |\cC(Q_i(u))| \le q_i$ by the definition of $I^0_u$. Otherwise, $I^0_u \le i-1$ for all $u \in Q_i(v)$, in which case by the inductive hypothesis we have that $|\cC(Q_{i-1}(u))| \le q_{i-1}$ for all $u \in Q_i(v)$.
Observe that if a vertex crosses $Q_i(v)$, then it crosses some $Q_{i-1}(u) \subset Q_i(v)$ with $u \in Q_i(v)$. In particular, $$\cC(Q_i(v)) \subset \bigcup_{u \in Q_i(v)} \cC(Q_{i-1}(u)) \, .$$ Since $Q_i(v)$ consists of $2^d$ cubes $Q_{i-1}(u)$ whose interiors are pairwise disjoint, we obtain that $|\cC(Q_i(v))| \le 2^d q_{i-1} = q_i$. 
\end{proof}

\begin{corollary}\label{cor:match}
Suppose that $R_v<\infty$ for all $v \in \Z^d$. Then there exists a matching $M \colon \Z^d \to \Pi$ such that $\|M(v)-v\| \le 3R_v$ for all $v \in \Z^d$.
\end{corollary}

\begin{proof}
By \cref{lem:construction of D}, $\cD$ satisfies \eqref{eq:D-prop1} and \eqref{eq:D-prop2}. By \eqref{eq:D-prop1} and \cref{prop:deterministic}, there is a matching $M:\mathbb Z ^d \to \Pi $ such that for all $v\in \mathbb Z ^d$ we have that $M(v)\in N_v$, where $N_v$ is the union of the box $Q_{I_v}(v)$ and the boxes in $\cD$ intersecting this box. By \eqref{eq:D-prop2}, the boxes that intersect $Q_{I_v}(v)$ are of side length at most $2R_v$, and therefore $\|M(v)-v\|\le 3R_v$. 
\end{proof}

\subsection{Probabilistic estimates}
\label{sec:cover-prob-estimates}

We now address the probabilistic part.
Suppose that $(\Pi_v)_{v \in \Z^d}$ is a perturbed lattice with perturbation law $\xi$. Recall that $B_r:=\{x\in\mathbb R ^d :\|x\|\le r\}$ is the $\ell ^{\infty }$ ball of radius $r$ and recall that
\[ h(r) := \P(\Pi \cap B_r = \emptyset) \qquad\text{and}\qquad p(r) := \P(\|\xi\| \ge r) .\]
We begin by finding the asymptotics of the hole probability $h(r)$.

\begin{lemma}[hole probability]
\label{lemma:hole_probability_Z_d}
There are constants $C,c>0$ such that for all $r$ large enough,
\[ p(r)^{Cr^d} \le h(r) \le p(r)^{cr^d}.\]
\end{lemma}

\begin{proof}
We have that
\begin{equation}\label{eq:product}
    h(r)=\prod _{v\in \mathbb Z ^d}\mathbb P (v+\xi \notin B_r).
\end{equation}
    To obtain the upper bound observe that
    \begin{equation*}
        h(r) \le \prod _{v\in B_{r/2}} \mathbb P \big( v+\xi  \notin B_r \big) \le p(r/2)^{|B_{r/2}|}\le p(r)^{cr^d},
    \end{equation*}
    where in the last inequality we used our assumption~\eqref{eq:assumption ii}.
    
    Next, we turn to prove the lower bound. To this end, we estimate the product in \eqref{eq:product} separately when $v\in B_{2r}$ and when $v\notin B_{2r}$. We have, using~\eqref{eq:assumption ii} again,
    \begin{equation}\label{eq:2.6}
        \prod _{v\in B_{2r}}\mathbb P \big( v+\xi \notin B_r \big) \ge p(3r)^{|B_{2r}|}\ge p(r)^{Cr^d}. 
    \end{equation}
    Next, note that for large $r$ and $v\notin B_{2r}$ we have that $\mathbb P (v+\xi \in B_r) \le p(r)\le 1/2$. Thus, using the inequality $1-x \ge e^{-2x}$ for $x \in [0,\frac12]$ we obtain
    \begin{equation}\label{eq:2.7}
        \prod_{v \notin B_{2r}} \P(v+\xi \notin B_r)
      \ge \exp \bigg( -2\sum_{v \notin B_{2r}} \P (v+\xi \in B_r) \bigg) \ge e^{-Cr^d},
    \end{equation}
    where in the last inequality we used that $\sum _{v\in \mathbb Z ^d}\mathbb P (v+\xi \in B_r)=\mathbb E |\Pi \cap B_r| \le Cr^d$.
    Substituting \eqref{eq:2.6} and \eqref{eq:2.7} into \eqref{eq:product} finishes the proof of the lower bound on $h(r)$.
\end{proof}

We now return to our construction and show that it yields a matching whose matching distance has optimal tail behavior, namely, the same as that of the hole probability.

\begin{lemma}\label{lem:tail}
We have that $R_v<\infty$ almost surely, and moreover, for all $r$ large enough,
\[ \P(R_v > r) \le h(r)^c.\]
\end{lemma}
\begin{proof}
We first bound the tail of $R^0_v$. Let $Q_r:=z+[0,r]^d\subseteq \mathbb R^d$ where $z\in (1/2,\dots ,1/2)+\mathbb Z ^d$, be a box of side length $r$. We would like to bound $\P (|\cC(Q_r)|>r^d)$. Let
\begin{equation*}
    L_0:= \big\{ u\in \mathbb Z ^d :d(u,Q_r^c)\ge \eps r \big\}\quad \text{and}\quad L_1:= \big\{ u\in \mathbb Z ^d :d(u,Q_r)\ge \eps r \big\},
\end{equation*}
where $\eps=\eps(d)>0$ is a constant depending only on $d$, chosen so that $|\mathbb Z ^d \setminus (L_0\cup L_1)| \le r^d/2$. Letting $N_i := |\cC(Q_r) \cap L_i|$ we have that $|\cC(Q_r)| \le r^d/2+N_0+N_1$ and therefore
\begin{equation}\label{eq:2.9}
   \P \big(|\cC(Q_r)|>r^d\big) \le \P \big(N_0>  r^d/4\big) + \P \big(N_1> r^d/4\big) . 
\end{equation}
We bound the two terms on the right-hand side on~\eqref{eq:2.9} separately. Starting with $N_0$, we sum over subsets $A\subseteq L_0$ with $|A|\ge r^d/4$ and obtain that
\begin{equation}\label{eq:2.10}
    \mathbb P \big(N_0\ge r^d/4\big) \le \sum _{A\subseteq L_0} \mathbb P \big(\forall v\in A, \ v+\xi _v  \notin B_r\big) \le 2^{r^d} p(\eps r)^{r^d/4} \stackrel{\eqref{eq:assumption ii}}{\le} p(r)^{cr^d},
\end{equation}
for $r$ sufficiently large.
We turn to bound the probability that $N_1$ is large, and start by bounding its expectation. We have
\begin{align*}
    \E[N_1] & = \sum _{v\in L_1} \P \big(v\in \mathcal C(Q_r)\big) \\ &= \sum _{v\in L_1} \P \big(0\in \mathcal C(-v+Q_r)\big) \le \E \big[ \big| \{v\in L_1 : (-v+Q_r)\cap [0,\xi ]\neq \emptyset \}\big|\big] \, .
\end{align*}
where here $[0,\xi ]\subseteq \mathbb R ^d$ is the interval connecting $0$ and $\xi $. The random variable inside the last expectation is zero when $\|\xi \|<\eps r$ since $d(L_1,Q_r)\ge \eps r$. Moreover, when $\|\xi \|\ge \eps r$, the number of boxes of the form $-v+Q_r$ with $v\in \mathbb Z ^d$ that intersect $[0,\xi]$ is at most $Cr^{d-1}\|\xi \|$, and we get
\begin{equation*}
    \E[N_1]\le Cr^{d-1} \mathbb E \big[\|\xi \|\mathds 1_{\|\xi \|\ge \eps r}\big] \le Cr^dp(r)^c.
\end{equation*}
In the last inequality, we used both our assumptions~\eqref{eq:assumption i} and~\eqref{eq:assumption ii}. 
Let us now bound the moment generating function of $N_1$. For any $\lambda >0$, we have
\begin{align}\label{eq:generating}
    \nonumber \E[e^{\lambda N_1}] &= \prod_{v \in L_1} \E[e^{\lambda \mathds{1}_{\{v \in \cC(Q_r)\}}}]  \\ &\le  \prod_{v \in L_1} \big[1+e^{\lambda }\mathbb P (v\in \mathcal C(Q_r)) \big] \le \exp\big(e^\lambda \E N_1 \big) \le \exp\big( Ce^\lambda r^{d}p(r)^c \big) .
\end{align}
Hence, taking $\lambda=\log (1/p(r)^c)$ where $c$ is the constant from the right-hand side of \eqref{eq:generating}, we obtain
\begin{equation}\label{eq:2.11}
   \P \big(N_1 > r^d/4\big) \le e^{-\lambda r^d/4} \, \mathbb E [e^{\lambda N_1}] \le 
\exp\big(-c\log (1/p(r))r^d +Cr^d\big) \le p(r)^{cr^d} \, . 
\end{equation}
Substituting \eqref{eq:2.10} and \eqref{eq:2.11} into \eqref{eq:2.9}, we arrive at the bound $$\P \big(|\cC(Q_r)|>r^d\big) \le p(r)^{cr^d} \, .$$ 
In view of~\eqref{eq:I_0}, we can take $r=2^{i}$ in the above and obtain
\[ \P\big(R^0_v > r \big) \le \P \big(|\cC(Q_i(v))|>r^d\big) \le p(r)^{cr^d}  \, , \]
for any $v\in \mathbb Z ^d$. Next,  
we turn to bound the tail of $R^1_v$. Using \eqref{eq:I_1}, we obtain
\[ \P\big(R^1_v > r\big) \le \sum_{ u\in \Z^d} \P\big(R^0_u > r + \tfrac14 \|u-v\|\big) \le \sum_{n=0}^\infty Cn^{d-1} \P\big(R^0_v > r + \tfrac14 n\big) \,  .\]
Since $p(r)$ is decreasing in $r$, the bound above gives
\[ \P\big(R^1_v > r\big) \le \sum_{n=0}^\infty Cn^{d-1} p(r)^{c(r+\frac14 n)^d} \le p(r)^{cr^d} \, ,\]
for all $r$ large enough. Finally,~\eqref{eq:I} yields that
\begin{align*}
 \P\big(R_v > r\big) &\le \sum_{u \in \Z^d} \P\big(R^1_u > \tfrac12 r,~ \|v-u\| \le R^1_u\big) \\ &\le \sum_{n=0}^\infty Cn^{d-1} \P\big(R^1_u > \min\{\tfrac12 r, n\}\big) \le p(r)^{cr^d} \, .
\end{align*}
In view of Lemma~\ref{lemma:hole_probability_Z_d}, the proof is complete.
\end{proof}

\begin{remark}\label{rem:assumptions}
    Following the above proof, we observe that without our assumptions~\eqref{eq:assumption i} and~\eqref{eq:assumption ii}, we obtain the bound $\P (\|M(v)-v\| \ge r) \le (E(cr)/r)^{cr^d}$, where $E(r) := \E[\|\xi\|\1_{\|\xi\| \ge r}]$.
    The assumptions are used in order to show that this bound is comparable to the hole probability $h(r)$, which for general distributions $\xi$, might not be the case. 
\end{remark}

\subsection{Proof of \cref{thm:invariant_matching_Z_d}}
\label{sec:thm-proof}

It follows from \cref{cor:match,lem:tail} that there exists a random matching $M$ between $\Z^d$ and $\Pi$ such that
\[ \P (\|M(v)-v\| \ge r) \le \mathbb P(3R_v\ge r) \le h(r)^c \qquad\text{for all }v \in \Z^d .\]
This almost gives \cref{thm:invariant_matching_Z_d}. All that is missing is the invariance of the matching and that the matching is perfect. The invariance can be obtained via a standard averaging argument, which we now explain.
Let $\mu$ be the probability measure on $(\R^d)^{\Z^d}$ corresponding to the above matching $M$.
For $n \ge 1$, define $\mu_n := \frac1{n^d} \sum_{v \in [n]^d} T_v \mu$, where $T_v$ is translation by $v$.
Observe that the sequence of measures $\{\mu_n\}_{n=1}^\infty$ is tight, since the tail bound on $\|M(v)-v\|$ is independent of $v$. Thus, there is a convergent subsequence, $\nu := \lim_{k \to \infty} \mu_{n_k}$. Finally, $\nu$ is invariant since for any given $u \in \Z^d$, we have that $T_u \nu = \lim_{k \to \infty} T_u \mu_{n_k} = \nu$, since $\|T_u \mu_n - \mu_n\| \le |[n]^d \setminus (u+[n]^d)| \cdot n^{-d}$.
This gives an invariant matching $M':\mathbb Z ^d\to \Pi $ with the same tail decay as $M$.  Next, we explain why such a matching $M'$ has to be perfect, that is, onto $\Pi$. For $v\in \Z^d$ we write $\Pi_v = v + \xi_v$ for the perturbed lattice point, and with a slight abuse of notation, identify $\Pi$ with the collection $(\Pi_v)_{v\in \Z^d}$. Then the function $\Upsilon:\Z^d\times \Z^d\to[0,1]$ given by
\[
\Upsilon(u,v) = \P\big(M'(u) = \Pi_v\big) 
\]
is diagonally invariant, meaning that $\Upsilon(u,v) = \Upsilon(u+w,v+w)$ for all $u,v,w\in \Z^d$. By the mass transport principle (see, for instance,~\cite[Lemma~8]{Holroyd-Pemantle-Peres-Schramm}) we have
\[
\sum_{v\in \Z^d} \Upsilon(u,v) = \sum_{v\in \Z^d} \Upsilon(v,u)\, , 
\]
which, since $M^\prime$ is a matching, implies that
\begin{equation*}
    1 = \E\Big[\sum_{v\in \Z^d} \mathbf{1}_{\{M^\prime(u) = \Pi_v\}} \Big] = \E\Big[\sum_{v\in \Z^d} \mathbf{1}_{\{M^\prime(v) = \Pi_u\}} \Big] = \P\big(\Pi_u\in M^\prime(\Z^d)\big) \, .
\end{equation*}
Since $\Pi$ is countable we get that almost surely $M^\prime$ is a perfect matching, which completes the proof of the theorem.
\qed

\section{Polynomial perturbations in $d=1$}
\label{sec:polynomial_in_d_1}

In this section we prove \cref{thm:d=1}.
This, we consider the point process $\Pi = \{v+\xi_v\, :\, v\in \Z \}$, where $(\xi_v)$ are i.i.d.\ perturbations with $\P[|\xi| \ge r] \simeq  r^{-\alpha}$ for some $\alpha\in(0,1)$. We start with the following simple lemma, in which we estimate the variance of the number of points in the interval $[0,t)$.

\begin{lemma}\label{lemma:d=1_variance}
    There are constants $C,c>0$ such that for all integers $t\ge 1$,
    \[
    ct^{1-\alpha } \le \var\big(\Pi[0,t) \big) \le Ct^{1-\alpha}.
    \]
\end{lemma}
\begin{proof}
    Since $\Pi[0,t)$ is a sum of independent indicators, we have 
    \[
    \var\big(\Pi[0,t) \big) = \sum_{k\in \Z} \P \big( k+\xi \in [0,t) \big)\cdot \P \big( k+\xi \notin [0,t) \big)  \, .
    \]
    For any integer $k\in [1,t/2]$, using the symmetry of $\xi$, we have that 
    \[
    ck^{-\alpha } \le \P\big(\xi < -k\big)\le \P\big(k+\xi \notin [0,t)\big) \le \P\big(|\xi | \ge k\big) \le Ck^{-\alpha}.
    \]
    Similarly, for any integer $k\in [t/2,t-1]$ we have
    \[
    c(t-k)^{-\alpha } \le  \P\big(k+\xi \notin [0,t)\big)  \le C(t-k)^{-\alpha}.
    \]
    We start with the lower bound on the variance. We have
    \begin{align*}
        \var\big(\Pi[0,t) \big) &\ge \sum_{k=1}^{\lfloor t/2\rfloor}  \P \big( k+\xi \in [0,t) \big)\cdot \P \big( k+\xi \notin [0,t) \big)  \\  &\ge c\sum_{k=1}^{\lfloor t/2\rfloor} \P(k+\xi \not\in [0,t]) \ge c\sum_{k=1}^{\lfloor t/2\rfloor} k^{-\alpha} \ge c t^{1-\alpha} \, ,
    \end{align*}
    which gives the desired lower bound. We turn to prove the upper bound. To this end, observe that by the symmetry of $\xi $ we have
    \begin{equation*}
        \sum _{k\in [0,t)} \P\big(k+\xi \notin [0,t)\big) = \sum _{k\in [0,t)} \sum _{m\notin [0,t)} \P\big(k+\xi \in [m,m+1)\big)= \sum _{m\notin [0,t)} \P\big(m+\xi \in [0,t)\big).
    \end{equation*}
    Thus, we obtain 
\begin{align*}
        \var\big(\Pi[0,t) \big) &\le \sum_{k\in [0,t)}  \P \big( k+\xi \notin [0,t) \big) + \sum_{k\notin [0,t)}  \P \big( k+\xi \in [0,t) \big) \\  &= 2\sum_{k\in [0,t)}  \P \big( k+\xi \notin [0,t) \big) \\ & \le 4 \sum_{k=0}^{\lfloor t/2\rfloor}  \P \big( k+\xi \notin [0,t) \big)  \le 4+4 \sum_{k=1}^{\lfloor t/2\rfloor} k^{-\alpha } \le Ct^{1-\alpha },
    \end{align*}
    as needed.
\end{proof}

We can now prove part (1) of Theorem~\ref{thm:d=1}.

\begin{proof}[Proof of Theorem~\ref{thm:d=1} (1)]
    This proof is inspired by \cite[Proof of Theorem 6(ii)]{Holroyd-Pemantle-Peres-Schramm}. 
    We construct the desired matching via the greedy algorithm that iteratively matches mutually closest points (see, e.g., \cite[discussion after Theorem~4]{Holroyd-Pemantle-Peres-Schramm}).
    This algorithm generates a perfect stable matching in the sense of Gale and Shapley \cite{gale1962college} under general conditions. Recall that a matching $M$ between two subsets $\Lambda,\Pi \subseteq \mathbb R$ is stable if for any $v\in \Lambda $ and $x \in \Pi \setminus \{ M(v)\}$ we have that $$|v-x|\ge \min \big\{ |v-M(v)|,|M^{-1}(x)-x|\big\} \, ,$$ so that it is not beneficial for both $v$ and $x$ to be paired together.
    It is well known that this works when $\Lambda \cup \Pi$ is discrete, non-equidistant, and has no descending chains (see, e.g., \cite[Lemma~15]{Holroyd-Pemantle-Peres-Schramm}). In our situation, $\Lambda=\Z$, the latter two properties do not hold, but it is easy to check that when $\xi $ is continuous, a suitable modification of them do hold and suffice for the construction. When $\xi$ has atoms, we might need to break ties during the greedy algorithm and we do this by matching left-most pairs first. 
    
    
    

    Next, we prove that the matching $M:\mathbb Z \to \Pi$ constructed above has the right tails. Let $t>0$ be a sufficiently large integer and define the set 
    \begin{equation}
        W:=\big\{ v\in [0,t) \cap \Z : |M(v)-v|>t \big\}.
    \end{equation}
    Let $a=\min W$ and $b=\max W$. We claim that for any $x\in \Pi \cap [a,b]$ we have that $M^{-1}(x)\in [a,b]$. To see this, suppose on the contrary that there is $x\in \Pi \cap [a,b]$ such that $M^{-1}(x)\notin [a,b]$. Without loss of generality suppose $M^{-1}(x)<a$. Then since $|M^{-1}(x)-x|>|a-x|$ and $|a-M(a)|>t>|a-x|$ the pair $a,x$ would be unstable, a contradiction.

    Thus, we obtain that 
    \begin{equation}\label{eq:W}
        |W| \le \max _{[r_1 ,r_2 ]\subseteq [0,t]} \big( r_2 -r_1 +1-\Pi [r_1 ,r_2 ] \big) \le 1+ \max _{r\in [0,t]} F(r)-\min _{r\in [0,t]} F(r),
    \end{equation}
    where $F(r):=r-\Pi [0,r)$.

    \noindent
    We finish the proof with a simple chaining argument in order to bound the expectation of the right hand side of \eqref{eq:W}. For all $r\in \mathbb N$ we have that $\mathbb E [F(r)]=0$ and by Lemma~\ref{lemma:d=1_variance}, $$\var (F(r))\le Cr^{1-\alpha } \, .$$ Thus, by Freedman's inequality (see, e.g., \cite[Theorems~3.6 and~3.7]{chung2006concentration}) we have for all $\lambda >0$ 
    \begin{equation}\label{eq:freedman}
        \mathbb P \big(  |F(r)|\ge \lambda  \big) \le 2\exp \Big( -\frac{\lambda ^2}{2\var (F(r))+\lambda } \Big) \le 2\exp \Big( -\frac{c\lambda ^2}{ r^{1-\alpha } +\lambda } \Big) \, .
    \end{equation}
    Next, let $j_1:=\lceil \log _2 t \rceil $ and fix $\epsilon := (1-\alpha) /4$. For $\lambda \ge 1$ define the event 
    \begin{equation*}
        \Omega := \bigcap _{j=0}^{j_1} \big\{ \forall u\in [0,t]\cap 2^j \mathbb Z , \  |F(u+2^j)-F(u)| \le \lambda t^{\frac{1-\alpha}{2}} 2^{ \epsilon (j-j_1)} \big\}.
    \end{equation*}
    Translation invariance implies that $F(u+r)-F(u)\overset{d}{=}F(r)$ and therefore by \eqref{eq:freedman} and a union bound we have  
    \begin{equation}\label{eq:bound omega}
        \mathbb P (\Omega ^c) \le C\sum _{j=0}^{j_1} t2^{-j} \mathbb P \big( |F(2^j)| > \lambda t^{\frac{1-\alpha}{2}} 2^{ \epsilon (j-j_1)} \big) \le C\sum _{j=0}^{j_1} t2^{-j} \exp    \big( -c\lambda 2^{\epsilon (j_1-j) } \big)\le Ce^{-c\lambda }.
    \end{equation}
    Moreover, we claim that on $\Omega $ we have $|F(r)|\le C\lambda t^{\frac{1-\alpha }{2}}$ for all $r\in [0,t]$. Indeed, for an integer $r\in [0,t]$, there is a sequence of integers $0=r_0\le r_1\le \cdots \le r_{j_1}=r$ such that $r_{j+1}\in 2^{j_1-j}\Z$ and $r_{j+1}-r_j \in \{0,2^{j_1-j}\}$ for all $j\in [0,j_1)$ (this sequence corresponds to the binary expansion of the integer $r$). On the event $\Omega $ we have that 
    \begin{equation*}
        |F(r)| \le \sum _{j=0}^{j_1-1} |F(r_{j+1})-F(r_j)| \le \sum _{j=1}^{j_1} \lambda t^{\frac{1-\alpha}{2}} 2^{ -\epsilon j} \le C \lambda t^{\frac{1-\alpha}{2}}.
    \end{equation*}
    We obtain using \eqref{eq:bound omega} that
    \begin{equation*}
     \mathbb P \Big(  \max _{r\in [0,t]}|F(r)|\ge C\lambda t^{\frac{1-\alpha }{2}}  \Big) \le \mathbb P (\Omega ^c) \le e^{-c\lambda } .
    \end{equation*}
    Thus, by translation invariance and \eqref{eq:W}  we have
    \begin{equation*}
        t \cdot \mathbb P \big( |M(0)| > t \big) =  \mathbb E |W| \le 1+2\cdot  \mathbb E \Big[  \max _{r\in [0,t]} |F(r)| \Big] \le  Ct^{\frac{1-\alpha }{2}}.
    \end{equation*}
    This finishes the proof part (1) of the theorem.
    \end{proof}

We turn to prove part (2) of Theorem~\ref{thm:d=1}. To this end, we will need the following simple claim.

\begin{claim}\label{claim:finite_moment_implies_L_1_truncation_cannot_grow_too_fast}
    Let $Z$ be a positive random variable such that $\E[Z^{\frac{1+\alpha}{2}}]<\infty$, then
    \[
    \lim_{t\to \infty}    t^{\frac{\alpha-1}{2}} \E\big[Z\wedge t\big] = 0 \, .
    \]
\end{claim}
\begin{proof}
    Observe that $t^{\frac{\alpha-1}{2}} (Z \wedge t) \le Z^{\frac{1+\alpha}2}$. Since $t^{\frac{\alpha-1}{2}} (Z \wedge t) \le t^{\frac{\alpha-1}{2}} Z \to 0$ as $t \to \infty$, almost surely, the dominated convergence theorem implies the claim.
\end{proof}
\begin{proof}[Proof of Theorem~\ref{thm:d=1} (2)]
  The proof here follows a similar argument given in~\cite[Proof of Theorem~2 (case $d=1$)]{Holroyd-Pemantle-Peres-Schramm}. Let $t\ge 1$ be a sufficiently large integer. We have
  \begin{align*}
      \E\Big[  \# \big\{ k\in \Z\cap [0,2t) \, : \, M(k)\not\in [0,2t) \big\} \Big]   &= \sum_{k=0}^{2t-1} \P\big(M(k) \notin [0,2t) \big)\\ &=\sum_{k=0}^{2t-1} \P\big(M(0) \notin [-k,2t-k) \big) \, . 
  \end{align*}
  Furthermore,
  \begin{multline*}
   \sum_{k=0}^{2t-1} \P\big(M(0) \notin [-k,2t-k) \big) \le \sum_{k=0}^{2t-1} \P\big( |M(0)| \ge k \wedge(2t-k)\big) \\ \le 2 \sum_{k=0}^t \P\big( |M(0)| \ge k\big)\le 2+2\int _0^t \mathbb P \big( |M(0)| \ge s\big) ds   \le 2+2 \, \E[|M(0)|\wedge t] \, ,
  \end{multline*}
  and hence
  \begin{equation}
    \label{eq:upper_bound_on_expexctation_d=1}
      \E\Big[  \# \big\{ v\in \Z\cap [0,2t) \, : \, M(v)\not\in [0,2t) \big\} \Big]  \le 2+ 2\, \E[|M(0)|\wedge t] \, .
  \end{equation}
  The random variable $\Pi[0,2t)$ is a sum of independent indicators with $\mathbb E [\Pi[0,2t)]=2t$ and by Lemma~\ref{lemma:d=1_variance} we have $\sigma _t^2:=\var (\Pi [0,2t))\ge ct^{1-\alpha }$. Thus, by the Lindenberg-Feller central limit theorem (see, e.g., \cite[Theorem~3.4.10]{durrett2019probability}) we have that $$\qquad \frac{\Pi[0,2t)-2t}{\sigma _t} \xrightarrow{ \  d \ } N(0,1) \, , \qquad \text{as } \ t\to\infty.$$
  We obtain for sufficiently large $t$ that
  \begin{multline*}
      \E\Big[  \# \big\{ v\in \Z\cap [0,2t) \, : \, M(v)\not\in [0,2t) \big\} \Big] \\  \ge \E\Big[\big(2t - \Pi[0,2t)\big)^+ \Big] \ge \sigma _t \cdot \mathbb P   \big( 2t - \Pi[0,2t) \ge \sigma _t \big) \ge ct^{\frac{1-\alpha}{2}} \, .
  \end{multline*}
  Combining the above with~\eqref{eq:upper_bound_on_expexctation_d=1}, we see that
  \[
  \liminf_{t\to \infty} \frac{\E[|M(0)|\wedge t]}{t^{\frac{1-\alpha}{2}}} > 0 \, , 
  \]
  which, in view of Claim~\ref{claim:finite_moment_implies_L_1_truncation_cannot_grow_too_fast}, finishes the proof of part (2) of Theorem~\ref{thm:d=1}. 
\end{proof}

\section{Open problems}
\label{sec:open_problems}
We conclude this paper with three open problems.
\subsection{Factor matching}
A natural question, raised already in the introduction, is whether one can construct a perfect \emph{factor} matching $M$ which attains the hole probability of the perturbed lattice $\Pi$? While we believe this should be possible, we note that allowing extra randomness can, in some settings, improve the tail behavior of the matching distance. See for example~\cite[Theorem~3]{Holroyd-Pemantle-Peres-Schramm} for a case in which this phenomenon occurs.
We also point out that the matching described in Theorem~\ref{thm:d=1} part (1) is in fact a factor, as can be verified directly from the construction.

\subsection{Other hyperuniform point processes}
As mentioned in the introduction, one of our motivations for studying this problem comes from recent works on random transports of hyperuniform point processes to the lattice (see Section~\ref{sec:related_works} for the relevant definitions and references). Besides random perturbations of the lattice, two other well studied two-dimensional hyperuniform point processes are the Ginibre ensemble and the zero set of the Gaussian Entire Function (GEF). The latter is defined as the zero set of the random Taylor series 
    \[
    F(z) = \sum_{n=0}^\infty \zeta_n\frac{z^n}{\sqrt{n!}}
    \]
    where $\{\zeta_n\}$ is a sequence of independent standard complex Gaussian random variables. The zero set of the GEF is translation invariant, and its hole probability scacles like $\exp(-cr^4)$, see~\cite{Nishry-hole,Sodin-Tsirelson-III}. In~\cite{Sodin-personal,Sodin-Tsirelson-II} (see also~\cite{Nazarov-Sodin-Volberg} for a related construction), it was shown there exists a translation invariant matching $M$ from the lattice to the GEF zeros such that 
    \[
    \P\big( \|M(0) \| \ge r\big) \le \exp\big(-cr^4/\log r\big) \, ,
    \]
    so the hole probability is nearly achieved in this case. We do not know whether the logarithmic correction can be removed for the GEF zeros, and it would be interesting to find out. 

    The Ginibre ensemble is the weak limit as $N\to \infty$ of the eigenvalues of an $N\times N$ random matrix with independent standard Gaussian entries, see~\cite[Section~4.3.7]{GAFbook}. One important feature of the Ginibre ensemble is the fact that it is a determinantal point process, meaning the its k-point functions can be expressed as determinants of a kernel. This structure makes many exact computations more tractable for this model. Its hole probability also scales like $\exp(-cr^4)$, as follows easily from Kostan's theorem~\cite[Theorem~4.7.3]{GAFbook}. Can one construct a translation invariant perfect matching which achieves this hole probability as a tail upper bound? We are not aware of any results in this direction in the literature.

\subsection{Polynomial perturbations in $d\ge 2$}
In dimensions $d\ge 2$, we have shown that for polynomial perturbations with finite mean, there exists a translation invariant perfect matching whose matching distance has optimal tail decay. Does this remain true for polynomial perturbations with infinite mean? We conjecture that for $d\ge 2$, and for all polynomial perturbations (with arbitrarily slow decay), there exists a translation invariant perfect matching which achieves the hole probability as a tail bound. 

This would contrast with the one-dimensional case, where Theorem~\ref{thm:d=1} shows there is a qualitative change of behavior for polynomial perturbations with infinite mean.

\bibliographystyle{abbrv}
\bibliography{matching}

\end{document}